\documentclass[12pt]{article}

\usepackage{amssymb,amsthm}
\usepackage{amsmath}
\usepackage{amsfonts}              
              
\usepackage{mathabx}
\usepackage{graphicx}
\usepackage{tikz}
\usetikzlibrary{arrows,chains,matrix,positioning,scopes}
\newcommand{\Z}{\mathbb Z}

\newcommand{\C}{\mathbb C}
\newcommand{\R}{\mathbb R}
\newtheorem{lem}{Lemma}[section]
\newtheorem{defn}[lem]{Definition}

\newtheorem{co}[lem]{Corollary}
\newtheorem{thm}[lem]{Theorem}
\newtheorem{prop}[lem]{Proposition}

\begin{document}

\title{Non Hyperbolic Free-by-Cyclic and One-Relator Groups}
\author{J.\,O.\,Button and R.\,P.\,Kropholler}

\newcommand{\Addresses}{{
  \bigskip
  \footnotesize

  J.\,O.\,Button, \textsc{Selwyn College, University of Cambridge,
Cambridge CB3 9DQ, UK}\par\nopagebreak
  \textit{E-mail address}: \texttt{j.o.button@cam.ac.uk}

  \medskip

  R.\,P.\,Kropholler, \textsc{Mathematical Institute,
    University of Oxford, Oxford OX2 6GG, UK}\par\nopagebreak
  \textit{E-mail address}: 
\texttt{Robert.Kropholler@maths.ox.ac.uk}

}}

\maketitle
\begin{abstract}
We show that the free-by-cyclic groups of the form $F_2\rtimes \Z$ act 
properly cocompactly on CAT(0) square complexes. We also show using 
generalised Baumslag-Solitar groups that all known groups defined by
a 2-generator 1-relator presentation are either SQ-universal or are
cyclic or isomorphic to $BS(1,j)$. 
Finally we consider free-by-cyclic groups which are 
not relatively hyperbolic with respect to any collection of subgroups.
\end{abstract}

\section{Introduction}
The recent far reaching work of Agol \cite{ag} and Wise \cite{wslng} 
proves that a word hyperbolic group $G$
acting properly and cocompactly on a CAT(0) cube complex 
must be virtually special, implying that $G$ has a
finite index subgroup which embeds in a right angled Artin group (RAAG).
A host of very strong conclusions then apply to the group $G$, of which 
the two that will concern us here are being linear (which we take to mean
over $\C$ although it is even true over $\Z$) and (if $G$ is not
elementary) being large, namely $G$ has a finite index subgroup 
surjecting to a non abelian free group.

However, if $G$ is a finitely presented non hyperbolic group acting 
properly and cocompactly on a CAT(0) cube complex, then the above consequences need no longer
hold, indeed $G$ can even be simple \cite{burgermozes}. Therefore
suppose we have a class of finitely presented groups which is believed
to be a well behaved class, but which contains both word hyperbolic and
non word hyperbolic examples. We can ask: first, do all examples in this
class have a nice geometric action, namely a proper cocompact action
on a CAT(0) cube complex, and second: do they all enjoy our
strong group theoretic properties which are a consequence of being a
virtually special group, namely being linear or being large. Note that
for non word hyperbolic groups, satisfying this geometric condition
will not necessarily imply these group theoretic properties.

In this paper the classes of groups we will be interested in are the
following three: groups of the form $F_k\rtimes_\alpha\Z$ for
$F_k$ a free group of finite rank $k$ and $\alpha$ an automorphism of $F_k$
(we refer to these as ``free-by-cyclic groups'');
the more general class of ascending HNN extensions $F_k*_\theta$
of finite rank free groups, where rather than $\theta$ having to be an
automorphism, as in the free-by-cyclic case, we allow $\theta$ to be
any injective endomorphism of $F_k$; and finally the class of groups
admitting a presentation with 2 generators and 1 relator, which we
refer to as 2-generator 1-relator groups. This last class neither contains
nor is contained in either of the other two classes but there is
considerable overlap.

In the free-by-cyclic case, it was recently shown in \cite{hawsn} that
such word hyperbolic groups do act properly and cocompactly on a CAT(0)
cube complex, and therefore are virtually special groups.
However Gersten in \cite{ger} gave
an example of a free-by-cyclic group which cannot act properly and
cocompactly on any CAT(0) space, so this result cannot hold in general in the 
non word hyperbolic case. Moreover \cite{bb},\cite{bridsonreeves} shows that 
there are free-by-cyclic groups which are not automatic, whereas groups that 
act nicely on CAT(0) cube complexes are automatic \cite{nibloreeves}. 
In Gersten's example the free group
has rank 3 but in Section 2 we consider
free-by-cyclic groups of the form
$F_2\rtimes_\alpha\Z$, none of which are word hyperbolic. Therefore
it is of interest to show directly that they act properly and
cocompactly on CAT(0) cube complexes, which we do in Section 2.
This work is based on unpublished work of Bridson and Lustig. 
Those authors give us the method of changing the natural topological model of 
the standard 2 complex (shown in Figure \ref{2complex}) to get rid of a 
pocket of positive curvature, and in this way they go on to show that these 
groups act on 2-dimensional CAT(0) complexes. 
We expand on this by showing that one can build these complexes from squares. 
This also strengthens a result of
T. Brady \cite{brady_complexes_1995}
who showed that there is a 2-complex of 
non-positive curvature made from equilateral 
triangles with fundamental group $F_2\rtimes_\alpha\Z$.

In Section 3 we consider 2-generator 1-relator groups. 
It is conceivable, but very definitely open, that a word
hyperbolic 2-generator 1-relator group always acts properly
and cocompactly on a CAT(0) cube complex
(for instance see \cite{wst} Conjecture 1.9).
However on moving to the non word hyperbolic case we see a different
picture emerging because a group acting properly and cocompactly on any
CAT(0) space cannot contain a Baumslag-Solitar group $BS(m,n)$ where
$|m|\neq |n|$. Thus
the examples of such nasty Baumslag-Solitar groups
as $BS(2,3)$ mean that we need not always have largeness nor linearity, or
even residual finiteness. In fact even restricting to residually
finite 2-generator 1-relator groups will not imply
linearity in general. This is because it was shown in \cite{weh} that
the group $\langle s,a,b|sas^{-1}=a^m,sbs^{-1}=b^n\rangle$ is not linear
over any field if $|m|,|n|>1$, and it was pointed out in \cite{drsp} that
the group
\[\langle t,a,b|tat^{-1}=b,tbt^{-1}=a^m\rangle\]
which is indeed a 2-generator 1-relator group
$\langle t,a|t^2at^{-2}=a^m\rangle$
that is known to be residually finite, 
contains this as an index 2 subgroup when $m=n$.

However, in looking for a ``large'' property that we hope is
held by all 2-generator 1-relator groups except for the soluble groups
$BS(1,m)$ and $\Z$, 
including the non residually finite groups, we
are led to the concept of a group $G$ being SQ-universal: namely that
every countable group embeds in a quotient of $G$. It was conjectured
by P.\,M.\,Neumann in \cite{pmn} back in 1973 that a non cyclic
1-relator group is either SQ-universal or isomorphic to $BS(1,m)$. Now
it was shown in \cite{sacsch} that a group having a 1-relator
presentation with at least 3 generators is SQ-universal, leaving
the 2-generator 1-relator case. Also \cite{ol} from 1995
showed that all non
elementary word hyperbolic groups are SQ-universal and this was generalized
to non elementary groups which are hyperbolic relative to any collection
of proper subgroups in \cite{amo} from 2007.

Recently the concept of a group being acylindrically
hyperbolic, which is more general than being
hyperbolic with respect to a collection of proper subgroups and
which implies SQ-universality, was
introduced in \cite{os} and studied in \cite{mios} where one
application was to 2-generator 1-relator groups. The authors divided
these groups into three classes with the first consisting of groups 
that they could show were acylindrically hyperbolic. We prove
in Theorem \ref{sqcs2} that all the groups in their second case, which
they show are not acylindrically hyperbolic,
are indeed SQ-universal unless equal to $BS(1,m)$.
In fact every group here is
formed by taking an HNN extension with base equal to 
a quotient of some free-by-cyclic group $F_k\rtimes_\alpha\Z$
for $\alpha$ finite order, along with infinite cyclic edge groups.
The proof proceeds by also identifying them as
generalized Baumslag-Solitar groups, whereupon we show more generally
in Theorem \ref{sqcs2} that any generalized Baumslag-Solitar group
either is SQ-universal or is isomorphic to $BS(1,m)$ or $\Z$.

This leaves their third case, which is exactly the class of 2-generator
1-relator groups that are ascending HNN extensions of finite
rank free groups.
Here we are not quite able to establish SQ-universality
of all of these groups not equal to $BS(1,m)$ or $\Z$,
though it is known to hold for the free-by-cyclic
case, but we do show in Corollary \ref{sq} that the only possible
exception would be a 2-generator 1-relator group equal to a strictly
ascending HNN extension of a finite rank free group which either fails
to be word hyperbolic and contains no Baumslag-Solitar subgroup, or
does contain a Baumslag-Solitar subgroup (but does not contain
$\Z\times\Z$) and  where all finite index subgroups have first Betti
number equal to 1. In both cases it is conjectured that no such
examples exist, so we have established P.\,M.\,Neumann's conjecture
for all the known 2-generator 1-relator groups.
We also obtain some general
unconditional statements,
such as Corollary \ref{comm} which says that
if the relator is in the commutator subgroup of $F_2$ then $G$ is
SQ-universal or equal to $\Z\times\Z$. 

In the final section we show that free-by-cyclic groups formed using an
automorphism $\alpha$ of polynomial growth are not hyperbolic relative
to any collection of proper subgroups (thus SQ-universality cannot
be established for all free-by-cyclic groups using only the result
of \cite{amo}), but are acylindrically hyperbolic unless $\alpha$ has
finite order as an outer automorphism. 

In our final class of ascending HNN extensions of finitely
generated free groups, it was shown
for the word hyperbolic
case in \cite{hgws} that $F_k*_\theta$ acts properly and cocompactly
on a CAT(0) cube complex if $\theta$ is an irreducible endomorphism, 
although that still	 
leaves the case where $\theta$ is not irreducible but
$F_k*_\theta$ is word hyperbolic. As for a non word hyperbolic group
of the form $F_k*_\theta$, 
there are some results on largeness for these groups in 
\cite{def1} which are not quite exhaustive
but make it likely that largeness holds throughout. However the example
mentioned above of the injective
endomorphism $\theta(a)=a^m,
\theta(b)=b^n$ of the rank two free group $F(a,b)$ for $m$ and $n$
both having modulus greater than 1, shows that
$F_k*_\theta$ need not be
linear over any field. Moreover the possible
existence in the non word
hyperbolic case of Baumslag-Solitar subgroups $BS(1,m)$ where $|m|\neq 1$
again means that such a group need not have a proper cocompact action on
any CAT(0) space. We finish by
showing in Theorem \ref{sap}
that the work in \cite{hawsn} and the Agol - Wise machinery
answers Problem 17.108 in the
recent edition of the Kourovka notebook \cite{kouonl}, namely that
Sapir's example of a strictly ascending
HNN extension of $F_2$ is indeed linear.  
 
We would like to thank Martin Bridson and Martin Lustig for allowing us to 
reproduce the results from \cite{unpub} here.

\section{Square complexes for free-by-cyclic groups in 
the rank 2 case}

\subsection{Consequences}

As mentioned before there are many nice consequences of word hyperbolic groups acting properly and cocompactly on CAT(0) cube complexes. However, no automorphism of $F_2$ is hyperbolic since they all fix the conjugacy class of $[a,b]^{\pm1}$, therefore, Agol's theorem does not apply and the complexes constructed are not known to be virtually special, though, due to the work of various authors
the groups are virtually special. 

However, there are still advantages to having a group act on a CAT(0) cube complex. For instance, abelian subgroups are quasi-isometrically embedded, and such groups are biautomatic \cite{gerstenshort},\cite{nibloreeves} and have a deterministic solution to the word problem in quadratic time \cite{epstein_word_1992}. 

Groups which act on CAT(0) square complexes have the further nice property that all of their finitely presented subgroups also act properly and cocompactly on CAT(0) square complexes. This is proved using a tower argument (see \cite{bridson_metric_1999},p. 217) and the fact that a sub complex of a non-positively curved square complex is itself a non-positively curved square complex (this may fail in higher dimensions). The construction also shows that for $F_2$-by-$\Z$ groups their geometric dimension is equal to their CAT(0) dimension, namely 2.

\subsection{Preliminaries}

We assume that the reader is familiar with the basics of CAT(0) geometry for which the standard reference is \cite{bridson_metric_1999}.

\begin{defn}
We say a metric space is {\em non-positively curved} if for each point there 
is a neighbourhood which is \textnormal{CAT(0)}.
\end{defn}

In the following we will study 2 dimensional piecewise euclidean 
($\mathbb{PE}$) complexes. These 
are complexes built from polygonal subsets of $\R^2$ by gluing along edges,
whereupon 
we put the natural path metric on the resulting complexes. For full details see 
\cite{bridson_metric_1999}.

Square complexes are special examples of $\mathbb{PE}$ complexes where all 
the cells are squares. 

The following theorems of Gromov \cite{grm} allow one to check 
whether a complex is non-positively curved just by looking at the links of 
vertices. 

\begin{thm}\cite{bridson_thesis}
A $\mathbb{PE}$ complex with finitely many isometry types of cells is non 
positively curved if and only if the link of each vertex is a \textnormal{CAT(1)} space. 
\end{thm}

In the two dimensional case the link of any vertex is a graph and so this can 
be reduced to the following. 

\begin{lem}\cite{bridson_metric_1999}
A graph is \textnormal{CAT(1)} if it contains no circuits of length less than $2\pi$.
\end{lem}

\begin{defn}
We say that an action is {\em proper} if for each compact set $K$ the set 
$\{g\in G:gK\cap K\neq \emptyset \}$ is finite. 
\end{defn}

As these groups will be the fundamental groups of non-positively curved 
spaces, they have an action on the universal cover. Since the spaces are 
compact the action will be proper and cocompact and it will also be a 
free action since these groups are torsion free. 

The groups $G_{\phi}$ that we shall be concerned with are mapping tori of $F_2 = F(x,y)$ by a single 
automorphism $\phi\in \mbox{Aut}(F_2)$. These groups have presentations of the form 
$$\langle x,y,t | txt^{-1}=\phi(x),tyt^{-1} = \phi(y)\rangle.$$

We start by considering the case of automorphisms which are of finite order. 

\begin{prop}
If $\phi\in \textnormal{Aut}(F_n)$ has order $q$ in $\textnormal{Out}(F_n)$
then $G_{\phi}$ is the fundamental group 
of a non-positively curved 2-complex. Furthermore, this is finitely covered by 
$\Gamma\times S^1$ where $\Gamma$ is a graph with fundamental group $F_n$.
\begin{proof}
Every finite order automorphism $\phi$ of $F_n$ can be 
realised as an isometry of a finite graph $\Gamma$; see for instance
\cite{cull} Theorem 2.1.
Let $X = \Gamma\times [0,1] /(x,0)\sim (\phi(x),1)$.

$X$ is locally isometric to $\Lambda\times(-\epsilon, \epsilon)$ where 
$\Lambda$ is a contractible subset of a graph. This will be CAT(0) and so 
$X$ is non-positively curved. 

If we take the cover corresponding to the obvious map to $\Z_q$ this will be 
$\Gamma \times S^1$.
\end{proof}
\end{prop}

We now want to look at automorphisms of infinite order. We require the 
following lemmas to ensure that we account for the general case with our 
construction. 

\begin{lem}
$G_{\phi}$ is defined up to isomorphism by $[\phi]\in \textnormal{Out}(F_2)$.
\end{lem}
\begin{proof}
If $\phi = ad_g\psi$, for $ad_g\in \mbox{Inn}(F_2)$ the inner automorphism conjugation by $g$, then 
\begin{align*}
G_\phi &\cong \langle x,y,t | txt^{-1}=\phi(x),tyt^{-1} = \phi(y)\rangle\\
 &\cong \langle x,y,t | txt^{-1}=g\psi(x)g^{-1},tyt^{-1} = 
g\psi(y)g^{-1}\rangle\\
  &\cong \langle x,y,t, t' | t'xt'^{-1}=\psi(x),t'yt'^{-1} = \psi(y), 
t' = g^{-1}t\rangle\\
   &\cong \langle x,y,t' | t'xt'^{-1}=\psi(x),t'yt'^{-1} = 
\psi(y)\rangle\cong G_{\psi}.
\end{align*}
\end{proof}

\begin{lem}
$G_{\phi}$ is defined up to isomorphism by the conjugacy class of $[\phi]\in 
\textnormal{Out}(F_2)$.
\end{lem} 
\begin{proof}
Let $\psi = \xi^{-1}\phi\xi$ then the following map defines an isomorphism.
\begin{align*}
\Omega:G_{\psi}&\to G_{\phi}\\
g&\mapsto \xi(g)\\
t&\mapsto t.
\end{align*}
\end{proof}

As such we will restrict to conjugacy in $\mbox{Out}(F_2) = \mbox{GL}_2(\Z)$.

In what follows we will need the following matrices:
$$F = \left(\begin{matrix}
  0 & 1 \\
  1 & 0 
 \end{matrix}\right),
 R = \left(\begin{matrix}
  1 & 1 \\
  0 & 1 
 \end{matrix}\right), L = \left(\begin{matrix}
  1 & 0 \\
  1 & 1 
 \end{matrix}\right).$$

\begin{lem}
Let $ \left(\begin{matrix}
  a & b \\
  c & d 
 \end{matrix}\right)=g\in \mbox{GL}_2(\Z)$ be a matrix. Then at least one of $g,-g,Fg$ or $-Fg$ is 
conjugate to a matrix with all non-negative coefficients. 
\end{lem}
\begin{proof}
We will split into 2 cases. First we will deal with the case of matrices with 
no entry equal to 0. Conjugating and multiplying by $F$, we may assume that 
$|a|\geq|b|,|c|,|d|$. We may now replace $g$ by $-g$ to make $a>0$. 
 
We note that every matrix of this form in $\mbox{GL}_2(\Z)$ cannot have one negative 
entry, as if this is the case then we see $|ad-bc|\neq1$.

We note that conjugation by $ \left(\begin{matrix}
  -1 & 0 \\
  0 & 1 
 \end{matrix}\right)$ changes the signs of $c$ and $b$. Using this we can 
assume that $a,c>0$, then we know that $b,d<0$ or $b,d>0$. If we are in the 
second case we are done, so we assume we are in the first case. 

We now conjugate by $R^{-1}$:
$$
R^{-1}gR = \left(\begin{matrix}  
1 & -1 \\  
0 & 1 
\end{matrix}\right)
\left(\begin{matrix}  
a & b \\  
c & d  
\end{matrix}\right)
\left(\begin{matrix}  
1 & 1 \\  
0 & 1  
\end{matrix}\right)
 = 
 \left(\begin{matrix}  
 1 & -1 \\  
 0 & 1 
 \end{matrix}\right)
 \left(\begin{matrix}  
 a & a+b \\  
 c & c+d  
 \end{matrix}\right).
$$
We know that $a,c, a+b$ are $>0$ so we see that $c+d\geq0$. Now 
$$
 \left(\begin{matrix}  
 1 & -1 \\  
 0 & 1 
 \end{matrix}\right)
 \left(\begin{matrix}  
 a & a+b \\  
 c & c+d  
 \end{matrix}\right)
  =
  \left(\begin{matrix}  
  a-c & a+b-(c+d) \\  
  c & c+d \end{matrix}\right)
 $$
so once again we know that $c,a-c,c+d$ are non-negative so $a+b-(c+d)\geq0$. 
 
If $g$ had an entry equal to 0 then either $g$ or $Fg$ is triangular. Further 
conjugating by $F$ and multiplying by $-1$ we can assume that $c=0$ and $a=1$. 
We now have two cases; 
namely $d=\pm1$. In the case where $d=1$ we can conjugate by 
$ \left(\begin{matrix}
  -1 & 0 \\
  0 & 1 
 \end{matrix}\right)$ to make $b\geq0$. If $d=-1$ then this matrix has order 2 
and we do not worry about this case.  
\end{proof}

The following lemmas are from \cite{cohen_what_1981}.

\begin{lem}
If $g$ is a non-diagonal matrix with all entries non-negative, then there is a 
subtraction of one row from another, which reduces the sum of the entries and 
produces a non-negative result.
\end{lem}

\begin{lem}
Let $g\in \mbox{GL}_2(\Z)$ with non-negative entries. Then there is a unique sequence of row subtractions which keep $g$ 
non-negative and reduce it to $I$ or $F$.
\end{lem}

We can think of these operations as multiplication by $R$ or $L$ which gives 
us the following result. 

\begin{co}
The semigroup generated by $-I,F,L$ and $R$ contains a conjugate of every 
infinite order matrix in $\textnormal{GL}_2(\Z)$.
\end{co}

As such we will only need to realise the automorphisms corresponding to these 
in our groups. 

\subsection{The construction}\label{sec2.3}

We will now construct non-positively curved complexes with $G_{\phi}$ as their 
fundamental groups, when $\phi$ is in the semigroup generated by
\begin{align*}
\lambda:a&\mapsto ba &\rho:a&\mapsto a\\
b&\mapsto b &b&\mapsto ab\\
\iota:a&\mapsto a^{-1} &\sigma:a&\mapsto b\\
b&\mapsto b^{-1} &b&\mapsto a.
\end{align*}

We see that from the above this gives us all $F_2$-by-$\Z$ groups. 
 
\begin{figure}
\center
\def\svgwidth{100mm}
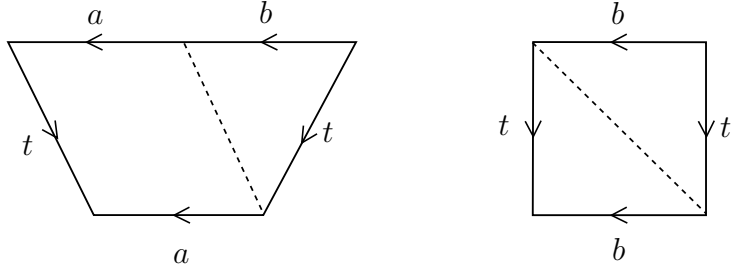
\caption{The 2-complex associated to $\lambda$}
\label{2complex}
\end{figure}

We start with the obvious 2-complex shown in Figure \ref{2complex} for the 
automorphism $\lambda$. This has a repeated corner which means it cannot 
support a metric of non-positive curvature.

\begin{figure}
\center
\def\svgwidth{118mm}
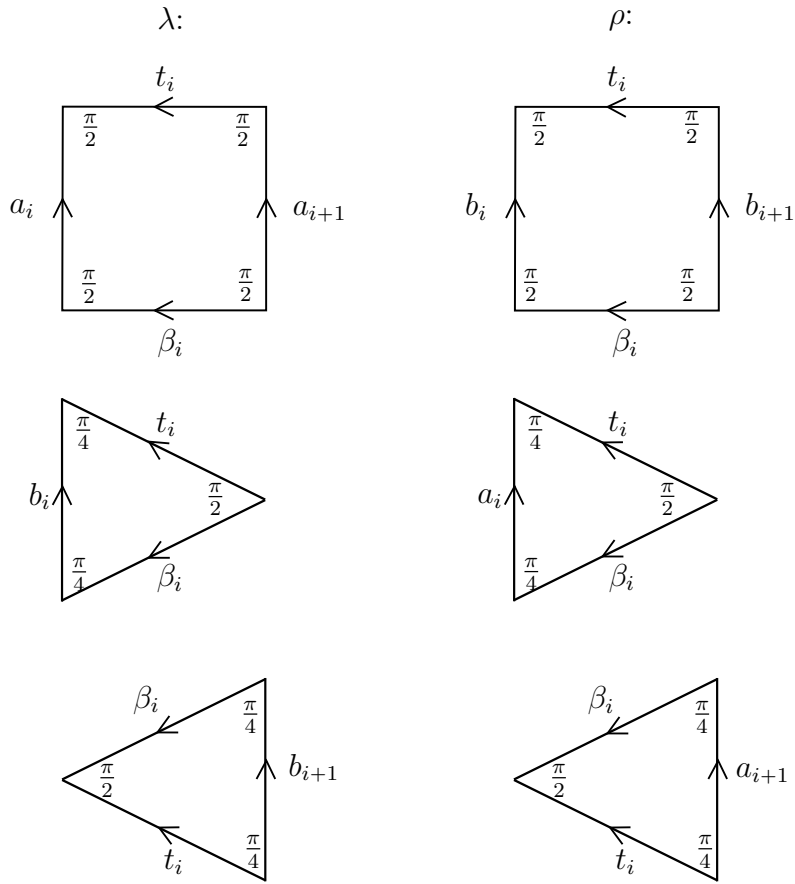
\caption{The basic building blocks for the construction with angles labelled}
\label{buildingblocks}
\end{figure}

To get rid of the repeated corner we cut our building blocks along the dotted 
line identifying the triangles with the repeated corner, resulting in our basic 
building blocks shown in Figure \ref{buildingblocks}.

By the above lemmas we can without loss of generality assume that our 
automorphism has the form $\phi = \eta_0\dots\eta_{n-1}\theta$ where $\eta_i = 
\rho$ or $\lambda$ and $\theta$ is one of the following finite order 
automorphisms:
\begin{align*}
\psi_1:(a,b)&\mapsto (a,b)\\
\psi_2:(a,b)&\mapsto (a^{-1},b^{-1})\\
\psi_3:(a,b)&\mapsto (b,a)\\
\psi_4:(a,b)&\mapsto (b^{-1},a^{-1}).
\end{align*}
We can assume that we only apply one of these and we do it at the end. This is 
because in $\mbox{Out}(F_2)$ the first and second give central elements 
whereas $\psi_4$ is equal to
the composition $\psi_2\psi_3$, but $\rho\psi_3 = \psi_3\lambda$.

We can now glue these together to get an automorphism which is in the class 
defined previously, so up to isomorphism we have all groups $G_{\phi}$. 
We do this in the following way: in Figure 2 let $i=0$ so that
we have the positive and the negative vertices $t_0^{\pm 1}$ at each end of
$t_0$. On performing the given gluing we have that $t_0^{\pm 1}$ are
not identified, but if we further stick these two vertices
together by identifying $a_0,b_0$ with $a_1,b_1$
respectively then the resulting
2-complex has fundamental group $F_2\rtimes_{\lambda}\Z$ or 
$F_2\rtimes_{\rho}\Z$. Now suppose that our automorphism
$\phi = \eta_0\dots\eta_{n-1}\theta$, 
where $\theta$ is one of the four special finite order
automorphisms above.
For each $i$ between 0 and $n-1$ we have a copy $C_i$ of the 2-complex
associated to either $\lambda$ or $\rho$ in Figure 2 which contains
the edge $t_i$. We then glue $C_i$ to $C_{i+1}$
 by the identity between the $a_{i+1},b_{i+1}$, which means that
the vertex $t_i^+$ is identified with $t_{i+1}^-$. Finally we glue
$C_{n-1}$ back to $C_0$ by identifying the $a_n,b_n$ with
$a_0,b_0$ so that $t_{n-1}^+$ becomes equal to $t_0^-$.  

We say a vertex is at {\em time $i$} if 
it is the vertex where $t_{i-1} $ and $t_i$ meet. 

We start with the case of vertices of time not equal to 0, thus
these will be where our complexes $C_{i-1}$ and $C_i$ are
glued together by the identity between the $a_i,b_i$.

The link of such a vertex is shown in Figure \ref{links}. We now want to 
assign angles such that there are no circuits of length less than $2\pi$.

If we assign angles as in Figure \ref{buildingblocks}, then we get two types of 
link as shown in Figure \ref{links}. Figure \ref{links} i) corresponds to the 
automorphisms at the $i$-th stage both being $\rho$ or both being $\lambda$. 
Figure \ref{links} ii) corresponds to when one automorphism is $\lambda$ and 
one is $\rho$.

\begin{figure}
\center
\def\svgwidth{118mm}
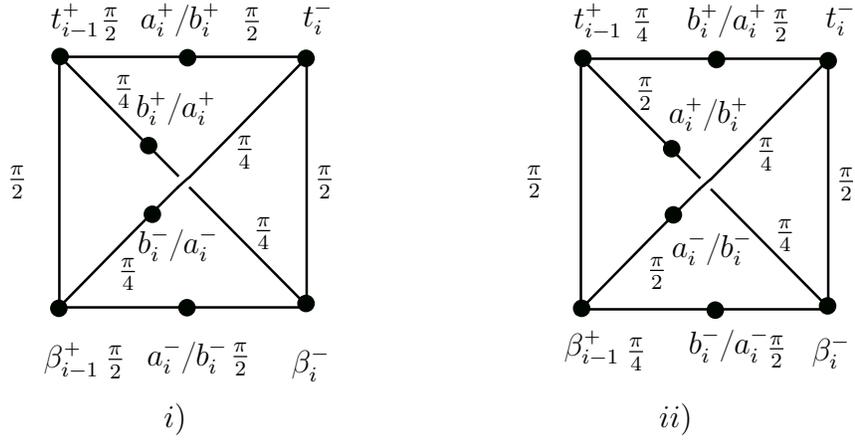
\caption{The possible links of a vertex.}
\label{links}
\end{figure}

We can see in either case that the link has no circuits of length less than 
$2\pi$. 

We now look at the case of the vertex at time 0. This will have to take into 
account the map $\theta$. We can consider the link as being split into 2 
halves as shown in Figure \ref{linkhalves}. The finite order maps defined 
earlier give vertex identifications. There are 16 possible links we may get in 
this way, corresponding to which automorphisms meet and to one of the 4 finite 
order automorphisms. All of these give a link which is homeomorphic to the 1 
skeleton of a tetrahedron with the set of angles depicted in 
Figure \ref{links}.

\begin{figure}
\center
\def\svgwidth{100mm}
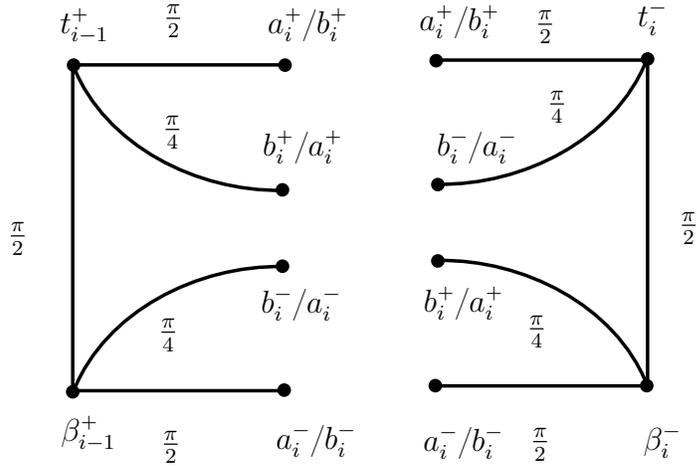
\caption{The 2 halves of a link.}
\label{linkhalves}
\end{figure}

\subsubsection{Square Complexes}

With a more careful assignment of angles we can see that the complexes above can be made into square complexes. 

\begin{figure}
\center
\def\svgwidth{135mm}
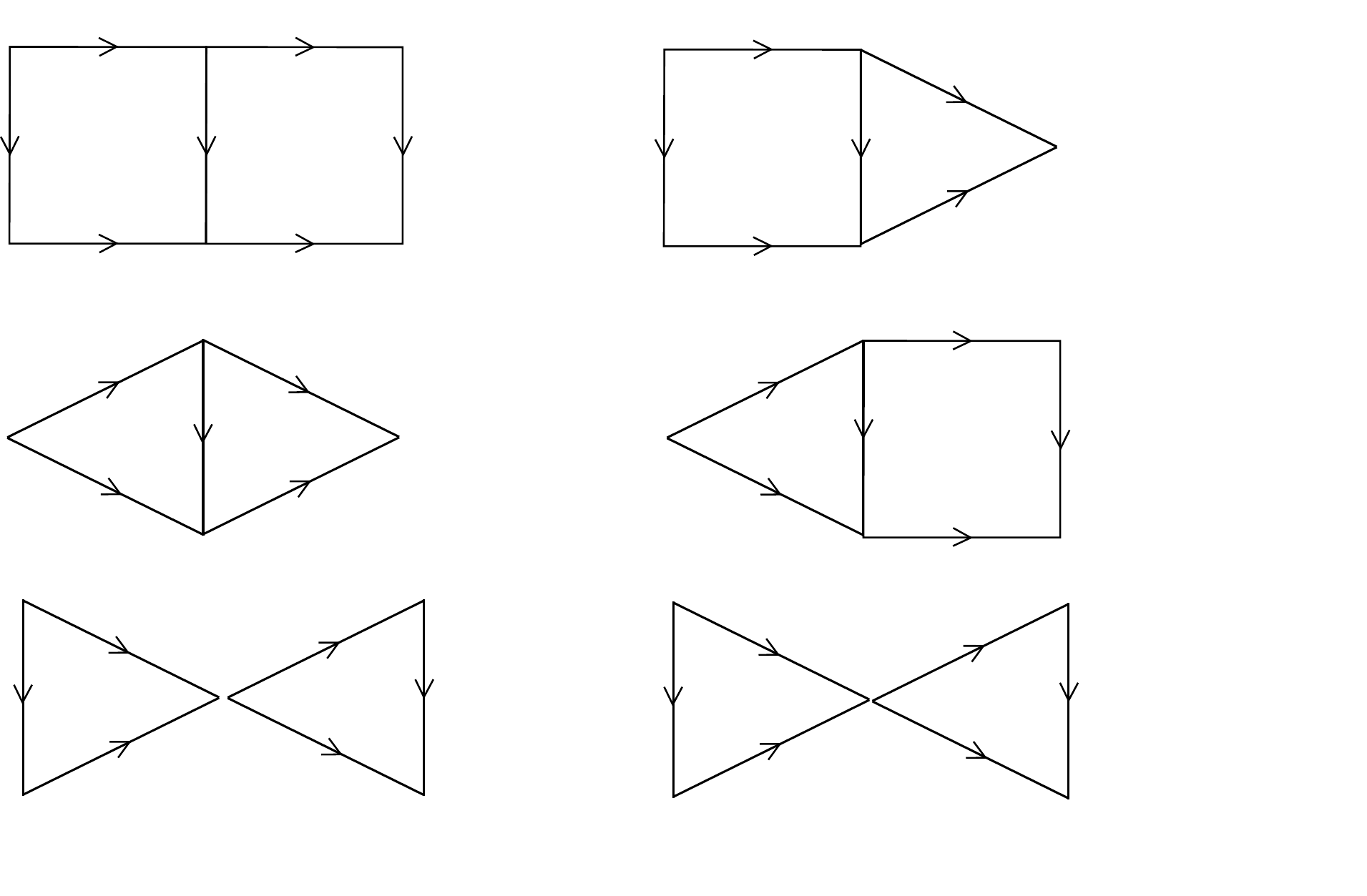
\caption{The 2 possible cases of automorphisms meeting at a vertex.}
\label{squares}
\end{figure}

We split the automorphism $\phi$ into one of three types depending on its 
decomposition in the semigroup described earlier:
\begin{enumerate}
\item $\phi= \rho^n$ or $\lambda^n$
\item $\phi= (\rho\lambda)^n(\rho\theta)^{\epsilon}$ or 
$(\lambda\rho)^n(\lambda\theta)^{\epsilon}$ where 
$\theta\in\{ \psi_3,\psi_4\}$ and $\epsilon\in\{0,1\}$
\item all other automorphisms.
\end{enumerate}

We will assign angles to our building blocks based on which meet at time $i$. 
We have depicted the two cases of our building blocks meeting at a vertex. 

In case 1 each time our building blocks meet, it will be of the type 
depicted in 
Figure \ref{squares} i). In this case we keep the angle assignment we had 
before and make the edges labelled $a_j$ or $b_j$ on the vertical sides of the 
two rectangles of length 2 and the other edges of length 1. We then make the 
edge between the 2 triangles of length $\sqrt{2}$ and subdivide the 
rectangles into squares of edge length 1, thus replacing the 
two triangles with a 
new building block which is the square formed by gluing them together. The 
link of the original vertices in this complex are depicted in 
Figure \ref{links2} i), where the edges corresponding to $a_i$ and $b_i$ 
have been suppressed as they have valence 2.

In case 2 each time our building blocks meet, it 
will be of the type depicted in 
Figure \ref{squares} ii). In this case we collapse the triangles to lines and 
then subdivide the resulting rectangles into 2 squares of side length 1. The 
link of the original vertices in this complex are depicted in 
Figure \ref{links2} ii), where the edges corresponding to $a_i$ and $b_i$ 
have been suppressed as they have valence 2.

\begin{figure}
\center
\def\svgwidth{80mm}
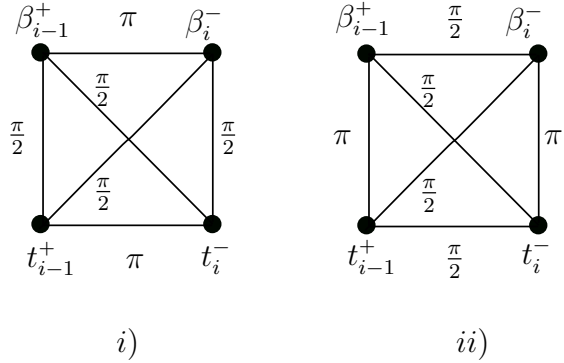
\caption{The links for automorphisms in cases 1 and 2.}
\label{links2}
\end{figure}

\begin{figure}
\center
\def\svgwidth{80mm}
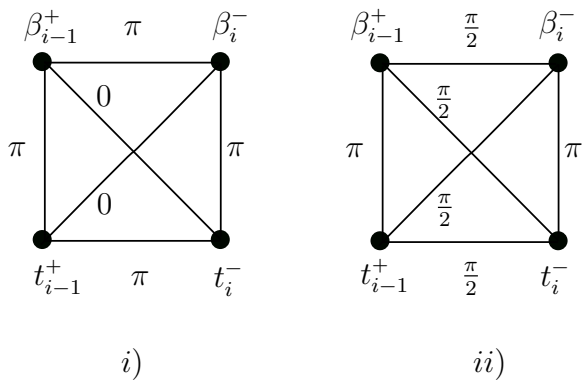
\caption{The links for automorphisms in cases 3.}
\label{case3links}
\end{figure}

In case 3 we will have a mix of both Figures \ref{squares}	i) and ii) and 
we collapse all the triangles to lines as we did for case 2. This  introduces 
degenerate squares from building blocks meeting as in 
Figure \ref{squares} i), 
which could affect the topology of the overall complex if there 
were a cylinder of such squares.
However here this will not happen as the only time a 
cylinder of degenerate squares could occur is if our automorphism is in case 1. 

The links of vertices in these complexes are depicted in 
Figure \ref{case3links}, where the edges corresponding to $a_i$ and $b_i$ have 
been suppressed as they have valence 2. The case of edges of 0 length are the 
degenerate squares where in fact the two edges become identified. 

In all the cases we see that the resulting complex will be a non-positively 
curved square complex.

\section{SQ universal groups}

A countable group $G$ is said to be SQ-universal (standing for
Subgroup Quotient) if every countable group can be embedded in
a quotient of $G$. This immediately implies that
$G$ contains a non abelian free group, and in turn is implied
by $G$ being large (having a finite index subgroup surjecting
to a non abelian free group). However an infinite simple group
$S$ containing $F_2$ would not be SQ-universal, nor would a
just infinite group such as $\textnormal{PSL}_n(\Z)$ for $n\geq 3$.

As for examples of groups which are SQ-universal, we have
all non elementary word hyperbolic groups by \cite{ol}.
This means that word hyperbolic groups with property (T)
provide lots of further
examples of groups which are SQ-universal but not large.
Moreover by \cite{amo} a finitely generated group
which is hyperbolic relative to any collection of proper subgroups
is SQ-universal (or virtually cyclic). 

An important class of groups in this area is 1-relator groups.
It was shown in \cite{bp} in 1978 that any group with a presentation
of deficiency at least 2 (thus any group having
an $n$-generator 1-relator presentation for $n\geq 3$) is large, 
leaving 2-generator 1-relator groups.
The question of when such a group $G$ contains $F_2$ has been
known for some time: yes, unless $G$ is isomorphic to a
Baumslag-Solitar group of the form $BS(1,n)$ (where $n\in\Z-\{0\}$)
or is cyclic.
Largeness is a different matter; for instance \cite{ep} showed that
the group $BS(m,n)$ is large if and only if $m$ and $n$ are not
coprime. In \cite{bcomp} we undertook extensive computation
suggesting that non large groups with a 2-generator 1-relator
presentation are few and far between, however there are more than
just Baumslag-Solitar examples. Moreover
we did not see a clean criterion that presented itself for
conjecture but in the case of SQ-universality there is a
statement that appeared in \cite{pmn} in 1973: a non cyclic
1-relator group is SQ-universal unless it is isomorphic to
$BS(1,n)$, thus if true this would be equivalent to containing $F_2$.

There have been a few results in this area since then; for instance
\cite{sacsch} showed in 1974 that a group with an $n$-generator 
1-relator presentation for $n\geq 3$ is SQ-universal (which was then
subsumed by the largeness result mentioned above). As for 2-generator
1-relator groups, Edjvet's thesis \cite{ephd} proves SQ-universality
in some useful cases. We showed in \cite{def1} Corollary 7.5
that it is true if the group is LERF but that is a very strong condition
to impose.

Recently the concept of a group being acylindrically hyperbolic was
introduced in \cite{os}. It holds if our group is non elementary  
and is relatively
hyperbolic with respect to a collection of proper subgroups. 
We will not need the definition here,
just the fact also in \cite{os} that such a group is SQ-universal.

This was followed up in \cite{mios} where the theory was applied to
various situations, including 1-relator groups to obtain the following.
Here a subgroup $H$ of a group $G$ is $s$-normal in $G$ if $H$ is
infinite and moreover $H\cap gHg^{-1}$ is infinite for all $g\in G$.
The relevance of this is that an $s$-normal subgroup of 
an acylindrically hyperbolic group must also be acylindrically
hyperbolic, so for instance $H\cong\Z$ being $s$-normal in $G$
implies that $G$ is not acylindrically hyperbolic  
(though it could certainly
be SQ-universal or even large).

\begin{prop} \label{1rel2gen} (\cite{mios} Proposition 4.20) 
Let $G$ be a group with two generators and one defining relator. Then at 
least one of the following holds:
\begin{itemize}
  \item[(i)] $G$ is acylindrically hyperbolic;
  \item[(ii)] $G$ contains an infinite cyclic $s$-normal subgroup.
  More precisely, either $G$ is infinite cyclic or it is an HNN-extension 
of the form $$G=\langle a,b, t\mid  a^t=b, w=1\rangle $$ of a 
2-generator 1-relator group
  $H= \langle a,b \mid w(a,b)\rangle $ with non-trivial center, so that 
$a^r=b^s$ in $H$ for some $r,s \in \Z \setminus\{0\}$. In the latter case $H$  
is (finitely generated free)-by-cyclic and contains a finite index normal 
subgroup splitting as a direct product  
of a finitely generated free group with an infinite cyclic group.
  \item[(iii)] $G$ is isomorphic to an ascending HNN extension
of a finite rank free group.
\end{itemize}
Moreover, the possibilities (i) and (ii) are mutually exclusive.
\end{prop}

Thus this establishes that groups in class (i) are SQ-universal. We will show
the same for groups in (ii) then discuss results for (iii).
When considering groups in case (ii), we will use the class of generalized
Baumslag-Solitar, or GBS, groups. These can be defined as those finitely
generated groups which act on a tree with infinite cyclic vertex and
edge stablisers. The two recent papers \cite{lgbs} and \cite{lbs}
cover a lot of ground in this area and we now mention the points we
will be using, referring to them for more detail.

We can describe a GBS group using the graph of groups theory,
where a finite graph (possibly with loops and/or multiple edges)
has labels consisting of a non zero integer at each end of each edge.
This label tells us the index of this edge group in the adjacent
vertex group, which determines the subgroup uniquely. In general
many different finite labelled graphs can give rise to isomorphic
GBS groups. One operation that can be performed without change of
the underlying group is an elementary collapse. This is when one
end of an edge $e$ next to a vertex $v$
is labelled $\pm 1$ and the edge is not a self loop.
We can then contract this edge and multiply all other labels next to
$v$ by the label at the other end of $e$. By doing this repeatedly,
we may assume that any edge with an end labelled by $\pm 1$ is a self loop.

Note that all GBS groups have deficiency 1, that is they admit a presentation
with one more generator than relator. In particular there always exists
a surjective homomorphism from any GBS group to $\Z$.

\begin{thm} \label{sqcs2}
If the group $G$ is as in case (ii) of the preceding Proposition
then $G$ is a generalized Baumslag-Solitar group. Moreover any
generalized Baumslag-Solitar group is either SQ-universal or
it is isomorphic to the Baumslag-Solitar group $BS(1,j)$ for
some $j\in\Z\setminus\{0\}$ or is infinite cyclic.
\end{thm}  
\begin{proof}
For the first part we can use Theorem C of \cite{krp}. This states
that the non cyclic finitely generated groups of cohomological
dimension 2 that have an infinite cyclic $s$-normal subgroup are
exactly the generalized Baumslag-Solitar groups. Now a 1-relator
group has cohomological dimension 2 if the relator is not a proper
power by \cite{lyn}, but a proper power gives rise to a group with
torsion, whereas the groups in case (ii) are all torsion free.

Now let $\Gamma$ be the underlying graph of our graph of groups that
results in the GBS group $G$. As mentioned above we assume that the only edge
labels equal to $\pm 1$ appear on self loops.

It is well known that if $\Gamma$ contains more than one cycle
(here we include self loops as cycles) then $G$ surjects to $F_2$
and so is SQ-universal, because we introduce a stable letter for
each cycle when forming $G$, and all vertex subgroups can be
quotiented out to leave only these stable letters which have no
relations between them.

It is also known that if $\Gamma$ is a tree then $G$ is virtually
$F_k\times\Z$ for $k\geq 2$ which is large, hence so is $G$. This can
be seen by quoting
Proposition 4.1 of \cite{lgbs} which
states that a group is a GBS group
with non trivial centre if and only if it is of the form 
$F_k\rtimes_\alpha\Z$ with $\alpha$ having finite order in $\mbox{Out}(F_k)$.
Now here $G$ will certainly
have a non trivial centre, namely an element
which is a common power of all the generators of the vertex subgroups
as these form a generating set for $G$. Moreover in the case of a tree
the surjective homomorphism $\theta$
from $G$ to $\Z$ has the property that
no non trivial element of a vertex (or edge) group lies in its kernel,
as if so then the whole vertex group does, thus so do the neighbouring
edge groups and so on across the whole tree. 

We now assume that $\Gamma$ has exactly one cycle $C$. First assume
this is not a self loop. We pick one edge $e$ lying in $C$ and remove
the interior of $e$ to form a tree $T$ and a group $H$ coming from
considering $T$ as the corresponding graph of groups. Thus we have
our homomorphism $\theta:H\twoheadrightarrow\Z$ as above,
with $G$ obtained from $H$ by taking generators $h_1,h_2$ of the
vertex groups at each end $v_1,v_2$ of $e$ and then adding a stable letter
$t$ which results in the presentation
\[G=\langle H,t|th_1^mt^{-1}=h_2^n\rangle\]
where $m,n$ are the labels at each end of $e$, neither of which are
0 or $\pm 1$. We now obtain a surjection from $G$ to a Baumslag-Solitar
group using the following folklore lemma:
\begin{lem}
Let $G$ be an HNN extension of the group $H$ amalgamating the subgroups
$A,B$ via the isomorphism $\phi:A\rightarrow B$. Suppose we have a
homomorphism $\theta$ from $H$ onto a quotient $Q$ with $\theta(A)$
isomorphic to $\theta(B)$ such that $\phi$ descends to an isomorphism
$\overline{\phi}$ from $\theta(A)$ to $\theta(B)$, meaning that
$\overline{\phi}$ is well defined and bijective with $\overline{\phi}
\theta=\theta\phi$. (This occurs if and only if $\phi(K)=L$ for
$K,L$ the kernels of the restriction of $\theta$ to $A,B$ respectively.)
Then on forming the HNN extension $R$ of $Q$ with stable letter $s$
amalgamating $\theta(A)$ and $\theta(B)$ via $\overline{\phi}$, we have
that the original HNN extension $G$ has this new HNN extension $R$ as a
quotient.
\end{lem}
\begin{proof}
We define a homomorphism from the free product $H*\langle t\rangle$
onto $R$ sending $t$ to $s$ and $h\in H$ to $\theta(H)\in Q$. We see
that this factors through $G$ because any relation in $G$ of the form
$tat^{-1}=\phi(a)$ has the left hand side mapped by $\theta$ to
$s\theta(a)s^{-1}$ and the right hand side to $\overline{\phi}\theta(a)$,
and these two things are equal in $R$ by the HNN construction.
\end{proof}

Consequently in our case we have $G$ surjects to $BS(k_1m,k_2n)$, where
$k_1=\theta(h_1)$ which is not equal to zero as mentioned above
because $H$ is formed from a tree,
and similarly for $k_2$. Now a Baumslag-Solitar group 
$BS(i,j)=\langle t,a|ta^it^{-1}=a^j\rangle$ is
known to be SQ-universal if neither of $i,j$ equal $\pm 1$, by Lemma
1.4.3 of \cite{ephd} if $i$ and $j$ are coprime, and by the well known
trick of setting $a^d$ equal to the identity when $d>1$ divides $i$ and
$j$, to get a surjection to $\Z*\Z_d$ which is virtually free otherwise.
As $|m|$ and $|n|$ are both greater than 1, we have that $G$
surjects to an SQ-universal group and so itself is SQ-universal.

We are now only left with the case where there is a single self loop
$L$ in our graph $\Gamma$, so that now $v_1=v_2$ and $h_1=h_2$.
The above proof also works here by removing $L$ this time,
unless
$|k_1|=|m|=1$ (or $|k_2|=|n|=1$ in which case we replace $BS(i,j)$ with
the isomorphic group $BS(j,i)$). If $L$ is all of $\Gamma$ then we
have that $G$ is just $BS(1,j)$ which is soluble, and so is a genuine
exception to being SQ-universal. Otherwise we note that $|k_1|=\pm 1$
implies that $h_1$ is mapped by $\theta$ to a generator of $\Z$.
On taking any edge $e_i$
not equal to $L$ with one endpoint $v_1$, let $a_i$ be the label
of $e_i$ at this end and $b_i$ the label at the other end
of $e_i$ next to the vertex group $\langle x_i\rangle$, say. We must
have $b_i$ dividing $a_i$ because of the relation $\theta(x_i^{b_i})=
\theta (h_1^{a_i})$ and the fact that $\theta(h_1)=\pm 1$. We now choose
a particular edge $e_i$ and remove the loop $L$ and the vertex $v_1$
from $\Gamma$, to form a possibly
disconnected graph that is a union of trees
$T_k$, with $T_1$ the tree containing the vertex group $\langle x_i\rangle$.

We then take a prime $p$ dividing $b_i$ and
consider the quotient of $G$ formed by setting $h_1$ and $x_i^p$ equal
to the identity. We have a surjective homomorphism with
domain the group obtained as a graph of groups from the tree $T_1$
and image $\Z_p$, consisting of the homomorphism to $\Z$
in the case of the tree $T_1$
and then composing with the map from $\Z$
to $\Z_p$.
We now extend this to a homomorphism from $G$ onto
$\Z*\Z_p$ which is SQ-universal as follows: send the stable letter $t$
to $1\in \Z$, so that the relation $th_1t^{-1}=h_1^n$ obtained from the
loop $L$ now has 
both sides sent to the identity. Moreover this holds for the relation
$x_i^{b_i}=h_1^{a_i}$ obtained from the edge $e_i$ as $p$ divides $b_i$.
Finally we send all vertex groups not in the component $T_1$ to the
identity, with the $x_j^{b_j}=h_1^{a_j}$ relations obtained from
the other edges $\{e_j:j\neq i\}$ that have endpoint $v_1$
all automatically satisfied.
\end{proof}
 
Note: this result can be compared to \cite{lgbs} Theorem 6.7 in which
the large GBS groups are determined, but there are cases for which the
graph $\Gamma$ is a single cycle where the group $G$ is SQ-universal
but not large.

We now come to case (iii), that of $G$ being equal to the HNN extension
$F_k*_{\theta}$, where $\theta:F_k\rightarrow F_k$ is injective but 
need not be surjective (if not then we call this a strictly
ascending HNN extension of $F_k$). Here we can
quote results of the first author in \cite{def1}. Theorem 5.4 of that
paper states that $G$ will be SQ-universal whenever
$\theta$ is an automorphism (unless $G\cong \Z,\Z\times\Z$ or the Klein bottle
group when the rank $k$ is 0 or 1). This is proved by showing that
$\Z\times\Z\leq G$ implies that $G$ is large, and then invoking
Ol'shanski\u{i}'s theorem on the SQ-universality of word hyperbolic groups
and the result in \cite{brink}
that not containing $\Z\times\Z$ and being word hyperbolic
are equivalent in the class of groups $F_k*_\theta$ when $\theta$ is an
automorphism.

In the case where $G$ is a strictly ascending HNN extension of a finite
rank free group, we have further results but they are not quite definitive.
Again we have that if $\Z\times\Z\leq G$ then $G$ is large (or is equal
to $\Z\times\Z$ or the Klein bottle group) by \cite{def1} Corollary 4.6.
However in the strictly ascending case there are examples where
$G$ does not contain $\Z\times\Z$
but does contain a Baumslag-Solitar subgroup, which must be of the form
$BS(1,m)$ for $|m|\neq 1$
so that $G$ fails to be word hyperbolic. In \cite{kap} it is conjectured
that a strictly ascending HNN extension of a finite rank free group is
is word hyperbolic if it does not contain a Baumslag-Solitar subgroup
and this conjecture seems to be widely believed, but a proof might well
require the machinery of train track maps to be developed in full for
injective endomorphisms of $F_k$. Moreover it is an open question whether
a 1-relator group (or indeed a group with a finite classifying space)  
containing no Baumslag-Solitar subgroups is word hyperbolic, so we would
be covered in our case if any of these (or their intersection) turned out to
be true.  

As for when $G$ contains $BS(1,m)$ for $|m|\neq 1$, \cite{def1} Theorem
4.7 states that either $G$ is large, or $G$ is itself a Baumslag-Solitar
group of the form $BS(1,n)$, or $G\not\cong BS(1,n)$ 
but $G$ has virtual first Betti number equal
to 1 and it is conjectured that the last case does not occur.   
Putting all this together, we have our result on the SQ-universality of
2-generator 1-relator groups:
\begin{co} \label{sq} 
If $G$ is a group given by a 2-generator 1-relator presentation
that is not $\Z$, $\Z\times\Z$ or the Klein bottle group 
then either $G$ is an SQ-universal group, 
or $G$ is a strictly ascending HNN extension $F_k*_\theta$ of a free group
$F_k$ which is not word hyperbolic and such that:\\
(i) either $G$ contains no Baumslag-Solitar subgroup (conjecturally
this does not occur)
or\\
(ii) $G$ contains a Baumslag-Solitar group $BS(1,m)$ for $|m|\neq 1$
but does not contain
$\Z\times\Z$ and the virtual first Betti number of $G$ is 1
(conjecturally this only occurs if $G\cong BS(1,n)$ for $|n|\neq 1$).
\end{co}

We finish this section with a couple of unconditional results.
\begin{co} \label{comm} If $G=\langle a,b|w(a,b)\rangle$ and $w$ is
in the commutator subgroup of $F(a,b)$ 
(and without loss of generality cyclically reduced)
then $G$ is SQ-universal, except when
$G\cong\Z\times\Z$ for $w$ a cyclic permutation of
$aba^{-1}b^{-1}$ or its inverse.      
\end{co}  
\begin{proof}
This proceeds by using the BNS invariant $\Sigma\subseteq S^1$ of $G$
in \cite{bns} and the proof is very similar to Theorem D of that paper.
The idea is that $\Sigma$ is an open subset of $S^1$ and if $\Sigma\cup
-\Sigma$ is not all of $S^1$ then we have a homomorphism 
$\chi:G\twoheadrightarrow\Z$ that expresses $G$ as a non-ascending
HNN extension. The Magnus decomposition of this extension is such that
$G$ is either in case (i) or case (ii) of Proposition \ref{1rel2gen}, so
that $G$ is SQ-universal by that theorem for case (i) or by
Theorem \ref{sqcs2} for case (ii).

Otherwise we have $\Sigma\cap-\Sigma\neq\emptyset$ as $S^1$ is connected,
which means that the kernel of $\chi$ is finitely generated and
consequently the 1-relator group
$G$ can be expressed as $F_k\rtimes_\alpha\Z$, which is SQ-universal if
$k\geq 2$. If $k=1$ then $\alpha$ is the identity as
$G$ surjects to $\Z\times\Z$, so $G=\Z\times\Z$ and therefore admits
only the above 2-generator 1-relator presentations by \cite{mks}
Theorem 4.11.
\end{proof}

Finally an SQ-universal group can be thought of as one with many
infinite quotients whereas an infinite residually finite group
can be thought of as having many finite quotients. We see that
all 2-generator 1-relator groups therefore have many quotients of some
kind:
\begin{co}
A group with a 2-generator 1-relator presentation is either
SQ-universal or residually finite.
\end{co}
\begin{proof}
This follows from Corollary \ref{sq} 
because \cite{bosp} proved that
a strictly ascending HNN extension of a finite rank free group
is residually finite (though as mentioned, not necessarily linear). 
\end{proof}

\section{Acylindrically hyperbolic mapping tori of free 
groups}

For the three cases in the last section, we had that the groups in case (i)
were all acylindrically hyperbolic whereas none in case (ii) were.
However when considering groups in case (iii) for SQ-universality, we did
this independently of results on acylindrically hyperbolic groups. It
can therefore be asked which mapping tori of finite rank
free groups are acylindrically hyperbolic and indeed this is exactly
Problem 8.2 in Section 8 of \cite{mios}. Moreover a solution just for
the 1-relator groups in this class would then completely determine
which 2-generator 1-relator groups are acylindrically hyperbolic, which
is their Problem 8.1.

It is clear that an ascending HNN extension $F_k\rtimes_\alpha\Z$
of $F_k$ formed using an automorphism $\alpha$ of finite order in
$\mbox{Out}(F_k)$ will not be acylindrically hyperbolic because of the existence
of an infinite order element in the centre. Thus a possible answer to
Problem 8.2 is that all other ascending HNN extensions of $F_k$ are
acylindrically hyperbolic with the exception of $BS(1,m)$ when $k=1$.
This would imply two mutually exclusive cases for these groups: either 
they are acylindrically hyperbolic or they are generalized 
Baumslag-Solitar groups, and it would also imply in answer to
Problem 8.1 
that a 1-relator
group is acylindrically hyperbolic if and only if it does not contain
an infinite cyclic $s$-normal subgroup. As a partial answer to Problem
8.2 we have
\begin{prop} \label{coho}
If a finitely generated group $G$ of cohomological dimension
2 has a finite index subgroup $H$ splitting over $\Z$ then either
$G$ is acylindrically hyperbolic or it is a generalized Baumslag-Solitar
group.
\end{prop}
\begin{proof}
We can apply \cite{krp} Theorem C to $H$ as it also has cohomological
dimension 2. This implies that
if $H$ splits over $A\cong\Z\leq H$ and
$A$ is $s$-normal then $H$ is a generalised Baumslag-Solitar
group and so is $G$ by \cite{krp} Corollary 3 (ii)
as it is torsion free. Otherwise we can
apply Corollaries 2.2 and 2.3 of \cite{mios} which state that if
$H$ is an amalgamated free product or HNN extension over an edge
group which is not $s$-normal and not equal to a vertex group under any 
inclusion then $H$ is acylindrically hyperbolic and therefore $G$ is
by \cite{mios} Lemma 3.8. However if $A\cong\Z$ is equal to a vertex
group then we have $H=BS(1,m)$ which is also a generalised
Baumslag-Solitar group.
\end{proof}
\begin{co} \label{ahyp}
An ascending HNN extension $F_k*_\theta$ for $k\geq 2$
that virtually splits over
$\Z$ is either acylindrically hyperbolic or is virtually
$F_k\times\Z$.
\end{co}
\begin{proof}  
All ascending HNN extensions of $F_k$ have geometric and thus
cohomological dimension 2, so by Proposition \ref{coho} we obtain
acylindric hyperbolicity unless we have a 
generalised Baumslag-Solitar group. They are all also residually
finite, but by \cite{lbs} Corollary 7.7
a generalised Baumslag-Solitar group is not residually finite
unless it is virtually $F_k\times\Z$ or $BS(1,m)$ when $k=1$.
\end{proof}

We now specialise to the case where the mapping torus is formed using
an automorphism, so we are back in the class of free-by-cyclic groups
$G=F_k\rtimes_\alpha\Z$, where we can say which such groups are
acylindrically hyperbolic. Of course this will be true if
$G$ is word hyperbolic or
hyperbolic with respect to a collection of proper subgroups.
There are plenty of examples of word hyperbolic free-by-cyclic
groups when $k\geq 3$. When $k=2$ there are none, but most are
relatively hyperbolic with respect to the peripheral $\Z\times\Z$
subgroup because they will be the fundamental group of a finite
volume hyperbolic 1-punctured torus bundle. The exceptions are
when the monodromy has finite order, giving the virtually
$F_2\times\Z$ case which cannot be acylindrically hyperbolic,
and parabolic monodromy where all groups will be
commensurable with $G=F(a,b)\rtimes_\lambda\Z$.

Here we consider the case where $[\alpha]\in \mbox{Out}(F_k)$ is a
polynomially growing automorphism, with recent results on this in
\cite{cl} which itself utilises the train track technology of Bestvina,
Feighn and Handel. The facts from \cite{cl} Section 5 that we need are:\\
\textbullet If $[\alpha]$ is polynomially growing then there is a positive
power $[\alpha^k]$ in UPG$(F_k)$, which is the subgroup of polynomially
growing outer automorphisms whose abelianised action on $\Z^k$ has
unipotent image.\\
\textbullet If $k\geq 2$ then every element of $\mbox{UPG}(F_k)$ has in its
class an automorphism $\alpha$ of $F_k$ such that either:\\
(i) There exists a non trivial $\alpha$-invariant splitting
$F_k=B_1*B_2$, so that $\alpha$ restricts to an automorphism of $B_1$
and also of $B_2$.\\
(ii) There exists a non trivial splitting $F_k=B_1*\langle x\rangle$, where
$B_1$ is $\alpha$-invariant and $\alpha(x)=xw$ for $w$ an element of $B_1$.\\ 
We also use the following two statements which follow immediately from
Corollary 4.22 and Theorem 1.4 in \cite{os04}:
\begin{lem}
If $G$ is a finitely presented group that is torsion free and
hyperbolic relative to
a collection of proper subgroups $\{H_1,\ldots H_l\}$ 
(the peripheral subgroups) then\\
(i) Any Baumslag-Solitar subgroup of $G$ is conjugate into some $H_i$.\\
(ii) Any $H_i$ is malnormal, so that if there is
 $g\in G$ with $H_i\cap gH_ig^{-1}$ non trivial then $g\in H_i$. Moreover
if there is $g\in G$ with $H_i\cap gH_jg^{-1}$ non trivial then $i=j$.
\end{lem}
\begin{thm} \label{poly}
If $[\alpha]\in\mbox{Out}(F_k)$ is a polynomially growing automorphism
then $G=F_k\rtimes_\alpha\Z$ is not hyperbolic
relative to any collection of proper subgroups, but it is acylindrically
hyperbolic unless $[\alpha]$ has finite order in $\mbox{Out}(F_k)$.
\end{thm}
\begin{proof}
By taking a power of $\alpha$, which corresponds to a finite index subgroup,
we can assume that $[\alpha]$ is in UPG$(F_k)$. Thus let us repeatedly
apply options (i) and (ii) until we have decomposed $F_k$ into a free
product of cyclic groups. 
These provide splittings over $\Z$ of this finite index subgroup and so
Corollary \ref{ahyp} immediately implies the second part of the statement.

We can picture this process of repeatedly splitting $G$ over $\Z$ 
as a finite rooted tree, as in \cite{cl} Lemma 5.10 which involves
describing it as a hierarchy. We have $F_k$ at the root vertex, with every
other vertex being labelled by a proper non trivial free factor
of $F_k$. To explain this, we split $F_k\rtimes_\alpha\Z$
using either Case (i) or Case (ii), which in general changes the automorphism
$\alpha$ but only to something equal in $\mbox{Out}(F_k)$ which we can
also call $\alpha$ for now.

If Case (i) is used then we have an $\alpha$-invariant
decomposition $F_k=B_1*B_2$, with stable letter $t$ inducing $\alpha$
by conjugation, which means that 
$F_k\rtimes_\alpha\Z=\langle t,F_k\rangle$ has been
split as an amalgamated free product 
$\langle t,B_1\rangle*_{\langle t\rangle}\langle t,B_2\rangle$ over
$\Z=\langle t\rangle$. We then draw two vertices labelled $B_1$ and $B_2$
as immediate descendents of the root vertex $F_k$.

If however case (ii) is used to split
$F_k=B*\langle x\rangle$ where $B$ has rank $k-1$ and $txt^{-1}=xw$ for
$w\in B$, then this corresponds to the HNN extension
$\langle t,F_k\rangle=\langle t,B\rangle*_{\langle x\rangle}$ where
$x^{-1}tx=wt$, so now $x$ is the stable letter of this HNN extension
conjugating the
infinite cyclic groups $\langle wt\rangle$ and $\langle t\rangle$.  
In this case we have only one immediate descendent vertex $B$ of
the root.

Having done this once, as the new vertex labels
$B_1,B_2$ or $B$ will be $\alpha$-invariant, we can take the appropriate
restriction of $\alpha$ and replace $\langle t,F_k\rangle$ by the
subgroup
$\langle t,B_1\rangle$, $\langle t,B_2\rangle$ or $\langle t,B\rangle$,
which is also free by cyclic but where the free part has smaller rank.
Moreover this restriction will also be a polynomially growing 
outer automorphism whose abelianised action has unipotent image, so
it is still in UPG. Thus we can then (on changing the restriction of
$\alpha$ by an inner automorphism)
apply Case (i) or Case (ii) again,
creating the next level of descendents and labelling them accordingly.
This continues  until we reach 
the leaf vertices (those with
no descendents), which are each labelled by an infinite cyclic free factor 
of $F_k$, so no further splitting occurs.

However each time we form the descendents of a vertex
by applying Case (i) or (ii) to the free factor of
$F_k$ labelling this vertex, we are 
liable to change the given
automorphism within the outer automorphism group
of this free factor. Now
we also need to keep track of the actual automorphisms and how they are
induced,
so there are further labels on each vertex (apart from the leaf
vertices) consisting of an automorphism of $F_k$ equal in $\mbox{Out}(F_k)$
to $\alpha$, along with a ``stable letter'' which is an element of 
$F_k\rtimes_\alpha\Z$ inducing this automorphism under conjugation.
This is defined inductively on the levels of the tree as follows: on
splitting our root vertex $F_k$, we have an automorphism equal to $\alpha$
in $\mbox{Out}(F_k)$ respecting this splitting. Now taking the 
composition of $\alpha$ with
an inner automorphism does not change $F_k\rtimes_\alpha\Z$, so we can assume
that $\alpha$ respects this splitting, whereupon we label the root vertex
$F_k$ with this $\alpha$ and the overall stable letter $t$ that induces
$\alpha$ by conjugation. This completes our labelling for the zeroth level.

Now suppose that $F_k=G_0,G_1,\ldots ,G_n$ are vertices on each successive
level leading to the vertex $G_n$, which is not a cyclic group as otherwise
it is a leaf vertex and therefore requires no further labelling. Therefore
$G_n$ has successor(s) $G_{n+1}$ in case (ii) (as well as
$G'_{n+1}$ if case (i) is taken). In order to describe the labels for
$G_{n+1}$ (and $G'_{n+1}$) if this too is not a leaf vertex, we suppose that
$G_n$ is already labelled by the automorphism
\[\gamma_n=\iota_{w_n}\iota_{w_{n-1}}\ldots \iota_{w_1}\alpha\mbox{ of }F_k,\]
where $w_i$ is an element of $G_i$ and $\iota_x$ stands for the inner
automorphism of $F_k$ consisting of conjugation by $x\in F_k$. Moreover
we assume that both $G_n$ and $G_{n+1}$ (and $G'_{n+1})$ are 
$\gamma_n$-invariant. We further assume that also labelling $G_n$ is 
the element $w_nw_{n-1}\ldots w_1t\in F_k\rtimes_\alpha\Z$ which we call
the stable letter $s_n$. 
Of course
if now $G_{n+1}$ (or $G'_{n+1}$) is cyclic then we do not need to label
these any further,
but otherwise we apply the splitting to $G_{n+1}$ under the automorphism
obtained by restricting $\gamma_n$ to $G_{n+1}$, thus we obtain an
automorphism of $G_{n+1}$ equivalent to $\gamma_n|_{G_{n+1}}$
in Out$(G_{n+1})$ which preserves $G_{n+2}$ (and $G'_{n+2}$) and which acts
accordingly in case (ii). Thus there is an element $w_{n+1}\in G_{n+1}$
such that 
we can define the automorphism $\gamma_{n+1}$ of $F_k$ to be
\[\iota_{w_{n+1}}\gamma_n=
\iota_{w_{n+1}}\iota_{w_n}\iota_{w_{n-1}}\ldots \iota_{w_1}\alpha,\]
thus $\gamma_{n+1}$ restricts to an automorphism 
of $G_{n+1}$ that preserves $G_{n+2}$ (and $G'_{n+2}$). We then 
further label the vertex $G_{n+1}$ by this automorphism $\gamma_{n+1}$
and the stable letter $s_{n+1}=w_{n+1}s_n$ that will induce this
map by conjugation.

Having completed the labelling, let us reverse the whole process and
build up the free-by-cyclic group $G$ from these splittings.
Let us start at a leaf vertex $G_{m+1}=\langle y\rangle$ of maximum
depth. As the automorphism $\gamma_m$ obtained from the vertex $G_m$
above preserves $G_{m+1}$ and $y$ is part of a free basis for $F_k$,
we must have $\gamma_m(y)=y$ as 
$[\gamma_m]=[\alpha]\in\mbox{UPG}(F_k)$ and so
all eigenvalues of the abelianised map equal 1. Thus $s_mys_m^{-1}=y$
giving us a copy of $\Z^2=\langle s_m,y\rangle$ in $G$, which must therefore
lie in a subgroup $C$ which is a conjugate
of one of the peripheral subgroups $H_i$ and so $C$ is also malnormal.
If $G_m$ gave rise in case (i) to $G_{m+1}=\langle y\rangle$ and 
$G'_{m+1}$, we would also have $G'_{m+1}=\langle z\rangle$ by 
maximality of depth. Thus $G_m=\langle y\rangle *\langle z\rangle$ and
the same stable letter $s_m$ commutes with
$z$, so we conclude by malnormality that $s_m,y,z\in C$. If however we had
case (ii) then again there is $z$ with $G_m=\langle y\rangle *\langle z\rangle$
but now we have $\gamma_m(y)=y$ and $\gamma_m(z)=zy^i$. Thus
$z^{-1}s_mz=y^is_m\in C$ so again $z\in C$ by malnormality. Hence in either
case we obtained from a leaf vertex a copy of $\Z^2$ lying in a conjugate
$C$ of a peripheral subgroup and shown that 
$\langle s_m, G_m\rangle$ is in $C$ too.

We now contract edges starting at the leaf vertices, one edge at a time if
the lower level vertex of this edge has one immediate descendent or a pair of 
vertices otherwise. 
We suppose that at each
leaf vertex $G_j$ of the current contraction of the tree,
we have $\langle s_j,G_j\rangle$ lying in a conjugate
$C$ of a peripheral subgroup (although $C$ could vary over the current leaf 
vertices). First suppose that $G_j$ is the only current descendent of 
$G_{j-1}$, so that $G_j$ was obtained by applying case (ii) to $G_{j-1}$
and the labelling indicates that $G_{j-1}=G_j*\langle x\rangle$, with
$\gamma_{j-1}$ sending $x$ to $xu$ for $u$ a
word in $G_j$. Thus $s_{j-1}xs_{j-1}^{-1}=xu$, but $s_{j-1}=w_j^{-1}s_j$ for
$w_j\in G_j$ so $s_{j-1}$ is already in $C$. Consequently 
$x^{-1}s_{j-1}x=us_{j-1}\in C$ implies that 
$x$ and therefore $\langle s_{j-1}, G_{j-1}\rangle$
is in $C$ as well.

Now suppose that there are other current descendents of $G_{j-1}$ in addition 
to $G_j$. This means that after further contraction if necessary there will
be the two leaf vertices $G_j$ and $G'_j$ which are the
descendents of $G_{j-1}=G_j*G'_j$, and we have by our assumption that 
$\langle s_j,G_j\rangle$ lies
in $C$ and similarly
$\langle s'_j,G'_j\rangle$ also lies in some conjugate $C'$ of a
peripheral subgroup. Now $s_{j-1}\in C$ as above, but the same argument also 
says that $s_{j-1}=(w'_j)^{-1}s'_j$ for $w'_j\in G'_j$. Thus the same element
$s_{j-1}$ is in both $C$ and $C'$, hence by malnormality they are equal
and we have $\langle s_{j-1},G_{j-1}\rangle\leq C$.
We can now continue contracting until we are left with the root,
concluding that all of $G=\langle t=s_0,G_0=F_k\rangle$  lies in $C$,
thus $G$ is not relatively hyperbolic with respect to proper subgroups.  
\end{proof}

This adds to results in the literature that provide a description
of free-by-cyclic groups according to the type of hyperbolicity:
for $G=F_k\rtimes_\alpha\Z$ we have that:\\
\textbullet $G$ is word hyperbolic if and only if no positive power of
$\alpha$ sends $w\in F_k\setminus\{id\}$ to a conjugate of itself.\\
\textbullet $G$ is relatively
hyperbolic if and only if $[\alpha]$ is not of polynomial growth. (The if
direction requires one to     
accept the results of  the unpublished manuscript \cite{dl} where the
peripheral subgroups are the mapping tori of the polynomially growing
subgroups under $[\alpha]$ of $F_k$, whereas the only if direction comes
from Theorem \ref{poly}.)\\
\textbullet $G$ is acylindrically hyperbolic but not relatively
hyperbolic if and only if $[\alpha]$ is polynomially growing and of infinite
order in $\mbox{Out}(F_k)$. (The same comment as above applies here, with
``if'' and ``only if'' reversed.)\\
\textbullet $G$ is not acylindrically hyperbolic if and only if
$[\alpha]\in\mbox{Out}(F_k)$ has finite order.
 
Note that all free-by-cyclic groups are known to be large by \cite{meq},
following \cite{def1}, \cite{1rel}
and \cite{hgws}, but are not known to be linear
unless $k=2$.

This leaves us with strictly ascending HNN extensions of finite
rank free groups, which we have seen can be less well behaved and indeed
need not be linear. However
we finish with an example to show that the recent work of Hagen and
Wise in \cite{hgws} on cubulation of ascending HNN extensions of
finitely generated free groups solves a problem in the Kourovka
notebook \cite{kouonl} on linearity of a particular group of this
kind. Problem 17.108
asks whether the group below is linear.

\begin{thm} \label{sap}
The group $\langle a,b,t|tat^{-1}=ab,tbt^{-1}=ba\rangle$ is
linear over $\Z$.
\end{thm}
\begin{proof}
The group is clearly a strictly ascending HNN extension of a finitely
generated free group
using the injective endomorphism $\theta(a)=ab,\theta(b)=ba$ of $F_2$.
Corollary 6.20 of \cite{hgws} states that this group is virtually
special, and hence linear over $\Z$, if $\theta$ is irreducible
and the HNN extension is word hyperbolic.
Here a reducible endomorphism $\theta$ of $F_k$ is one where
there is a free product decomposition $F_k=F_{k_1}*\ldots *F_{k_r}*C$
for $1\leq r\leq k$ and where each $F_{k_i}$ is non trivial (though
$C$ might be in which case $2\leq r$)
such that $\theta(F_{k_i})$ is sent into a conjugate
of $F_{k_{i+1}}$ with $i$ considered modulo $r$.

Theorem A in \cite{kap} considered the word hyperbolicity of
strictly ascending HNN extensions of a finite rank free group $F(X)$ 
and came up
with a result when the endomorphism $\theta$
is an immersion, which is defined to mean
that for all $x,y\in X\cup X^{-1}$ with $xy\neq e$, the word 
$\theta(x)\theta(y)$ admits no cancellation (so called because the standard
selfmap of the $|X|$-petalled rose is an immersion). 
This implies that if
$w=x_1\ldots x_n$ is written as a reduced word on $X$ for $x_i\in
X\cup X^{-1}$ then no cancellation can occur on the right hand side of
$\theta(w)=\theta(x_1)\ldots \theta(x_n)$. Thus the word length of
$\theta(w)$ is the sum of that for the $\theta(x_i)$. This theorem states that
if there is no periodic conjugacy class, meaning that there is no
$w\in F(X)-\{e\}$ and  $i,j>0$ such that $\theta^i(w)$ is conjugate to
$w^j$ in $F(X)$, then the strictly ascending HNN extension $F(X)*_\theta$
is word hyperbolic.
We can apply this to our group as
$\theta$ is indeed
an immersion and the fact that each generator maps to a word of
length two implies that the word length of $\theta^i(w)$ is $2^i$ times
that of $w$.
But we can assume by conjugating that $w$ is cyclically reduced, in which
case $\theta^i(w)$ is also cyclically reduced, as will be $w^j$ whose
word length is $j$ times that of $w$. Now
if $\theta^i(w)$ and $w^j$ are both cyclically reduced and are conjugate
in $F_2$ then they must have the same length and be cyclic permutations
of each other, thus $j=2^i$. 

Next suppose the word length of $w$ is even and set $w=uv$ for $u,v$ both
half the length of $w$. We have $\theta^i(u)\theta^i(v)$ is a cyclic
permutation of $w^{2^i}$ and so is equal to $(sr)^{2^i}$
without cancellation in this expression, where $w=rs$ (but $r$ and $s$ are
not necessarily of equal length). Thus 
$\theta^i(u)=(sr)^{2^{i-1}}=\theta^i(v)$. 
As an immersion must be injective, we obtain $u=v$ and $\theta^i(u)$
is a cyclic permutation of $w^{2^{i-1}}=u^{2^i}$.
Thus we can
continue to cut $w$ in half at each stage until it has odd length, whilst
still preserving the fact that $\theta^i(w)$ is a cyclic conjugate of
$w^{2^i}$. 
When we reach this point, we
consider the number of appearances of $a^{\pm 1}$ and $b^{\pm 1}$ in $w$, which
cannot be equal, thus nor can it be equal for $w^{2^i}$ or for any cyclic
conjugate thereof.
But as $\theta(a)$ and
$\theta(b)$ both have equal exponent sums, this is true for any element in
the image of $w$ which is a contradiction.

Now showing irreducibility of an endomorphism is straightforward in
rank 2 because any proper non trivial
free factor is just the cyclic group generated by
a primitive element. Thus to show $\theta$ is irreducible it is enough
to rule out there being a primitive element $x$ such that $\theta(x)$
or $\theta^2(x)$ is conjugate to $x^j$ for some $j\neq 0$. 
But this has been eliminated for $j>0$, and also
for $j<0$ by considering $\theta^2(x)$ and $\theta^4(x)$ respectively.
\end{proof}

\Addresses


\begin{thebibliography}{99}

\bibitem{ag} I.\,Agol,
{\it The virtual Haken conjecture},
Doc. Math. {\bf 18} (2013), 1045--1087. 


\bibitem{amo}
G.\,Arzhantseva, A.\,Minasyan, D.\,Osin, 
{\it The $SQ$-universality and residual properties of relatively hyperbolic groups} 
{J. Algebra} {\bf 315} (2007), no. 1, 165-177.

\bibitem{bp} B.\,Baumslag and S.\,J.\,Pride,
{\it Groups with two more generators than relators},
J. London Math. Soc. {\bf 17} (1978) 425--426.

\bibitem{bns} R.\,Bieri, W.\,D.\,Neumann and R.\,Strebel,
{\it A geometric invariant of discrete groups},
Invent. Math. {\bf 90} (1987) 451--477.

\bibitem{bosp} A.\,Borisov and M.\,Sapir,
{\it Polynomial maps over finite fields and residual finiteness of mapping
tori of group endomorphisms},
Invent. Math. {\bf 160} (2005) 341--356.

\bibitem{bb} N.\,Brady and M.\,R.\,Bridson, 
On the absence of biautomaticity in certain graphs of abelian groups, preprint.

\bibitem{brady_complexes_1995} T.\,Brady,
{\it Complexes of nonpositive curvature for extensions of $F_2$ by 
$\Z$},
Topology and its Applications {\bf 63} (1995)
	267--275.
	
\bibitem{bridson_thesis} M.\,R.\,Bridson,
Geodesics and curvature in metric simplicial complexes. 
{\it Group theory from a geometrical viewpoint}, 
373--463, World Sci. Publ., River Edge, NJ, 1991.


\bibitem{bridson_metric_1999} M.\,R.\,Bridson, A.\,Haefliger,
Metric spaces of non-positive curvature,
Springer-Verlag, Berlin, 1999.

\bibitem{unpub} M.\,R.\,Bridson, M.\,Lustig
{\it $F_2$-by-cyclic groups and CAT(0) complexes},
Unpublished Manuscript 1997.

\bibitem{bridsonreeves} M.\,R.\,Bridson and L.\,Reeves, 
On the absence of automaticity in certain free-by-cyclic groups, 
in preparation. 
	
\bibitem{brink} P.\,Brinkmann 
{\it Hyperbolic automorphisms of free groups},
Geom. Funct. Anal. {\bf 10} (2000) 1071--1089.

\bibitem{burgermozes} M.\,Burger, S.\,Mozes,
{\it Finitely presented simple groups and products of trees}, 
C. R. Acad. Sci. Paris S\/{e}r. I Math. {\bf 324} (1997), no. 7, 747--752.

\bibitem{def1} J.\,O.\,Button,
{\it Large groups of deficiency 1},
Israel J. Math. {\bf 167} (2008) 111--140.

\bibitem{1rel} J.\,O.\,Button,
{\it Largeness of LERF and 1-relator groups},
Groups Geom. Dyn. {\bf 4} (2010), 709--738.

\bibitem{bcomp} J.\,O.\,Button, 
{\it Proving finitely presented groups are
large by computer}, Experimental Math. 
{\bf 20} (2011), 153--168.

\bibitem{meq} J.\,O.\,Button,
{\it Free by cyclic groups are large},\\
\texttt{http://arxiv.org/1311.3506}

\bibitem{cl} C.\,H.\,Cashen and G.\,Levitt,
{\it Mapping tori of free group automorphisms, and the Bieri-Neumann-Strebel
invariant of graphs of groups},\\
\texttt{http://arxiv.org/1412.8582}

\bibitem{cohen_what_1981} M.\,Cohen, W.\,Metzler and A.\,Zimmermann,
{\it What does a basis of F(a, b) look like?}
{Mathematische Annalen} {\bf 257} (1981),
	435--445.

\bibitem{cull} M.\,Culler,
Finite groups of outer automorphisms of a free group.
{\it Contributions to group theory}, 197--207,
Contemp. Math. {\bf 33}, Amer. Math. Soc., Providence, RI, 1984.

\bibitem{dl} F.\,Gautero and M.\,Lustig,
{\it The mapping-torus of a free group automorphism is hyperbolic relative
to the canonical subgroups of polynomial growth},
\texttt{http://arxiv.org/0707.0822}

\bibitem{drsp} C.\,Dru\c tu and M.\,Sapir,
{\it Non-linear residually finite groups},
J. Algebra {\bf 284} (2005), 174--178.

\bibitem{ephd} M.\,Edjvet, 
{\it The concept of ``Largeness'' in Group Theory},
Ph.\,D thesis, University of Glasgow (1984)

\bibitem{ep} M.\,Edjvet and S.\,J.\,Pride,
{\it The concept of ``largeness'' in group theory II}, in
Groups -- Korea 1983, Lecture Notes in Math. {\bf 1098}, Springer, Berlin,
1984, pp.\,29--54.

\bibitem{epstein_word_1992} D.\,B.\,A.\,Epstein et al.,
Word processing in groups,
Jones and Bartlett Publishers, Boston, 1992.

\bibitem{ger} S.\,M.\,Gersten,
{\it The automorphism group of a free group is not a CAT(0) group},
Proc. Amer. Math. Soc. {\bf 121} (1994) 999-1002.

\bibitem{gerstenshort} S.\,M.\,Gersten, H.\,B.\,Short,
{\it Small cancellation theory and automatic groups.}
Invent. Math. {\bf 102} (1990), no. 2, 305--334. 


\bibitem{grm} M.\,Gromov, Hyperbolic groups. {\it Essays in group theory},
75--263, Math. Sci. Res. Inst. Publ. {\bf 8}, Springer, New York, 1987.

\bibitem{hgws} M.\,F.\,Hagen and D.\,T.\,Wise,
{\it Cubulating hyperbolic free-by-cyclic groups:
the irreducible case},
\texttt{http://arxiv.org/1311.2084}

\bibitem{hawsn} M.\,F.\,Hagen and D.\,T.\,Wise,
{\it Cubulating hyperbolic free-by-cyclic groups: the general case},
\texttt{http://arxiv.org/1406.3292}

\bibitem{kap} I.\,Kapovich, {\it Mapping tori of endomorphisms of free groups},
Comm. Algebra {\bf 28} (2000) 2895--2917. 

\bibitem{kouonl}
{\it Unsolved problems in group theory.
The Kourovka notebook. Number 18 (English version)},
Edited by V.\,D.\,Mazurov and E.\,I.\,Khukhro.
\texttt{http://arxiv.org/1401.0300}

\bibitem{krp} P.\,H.\,Kropholler,
{\it Baumslag-Solitar groups and some other groups of
cohomological dimension two},
Comment. Math. Helvetici {\bf 65} (1990) 547--558.

\bibitem{lgbs} G.\,Levitt,
{\it Generalized Baumslag-Solitar groups: rank and finite index subgroups},
\texttt{http://arxiv.org/1304.7582}

\bibitem{lbs} G.\,Levitt,
{\it Quotients and subgroups of Baumslag-Solitar groups},
\texttt{http://arxiv.org/1308.5122}

\bibitem{liu} Y.\,Liu,
{\it Virtual cubulation of nonpositively curved graph manifolds}, 
J. Topol.
{\bf 6} (2013), 793--822.

\bibitem{lyn} R.\,C.\,Lyndon,
{\it Cohomology theory of groups with a single defining relation},
Ann. of Math. {\bf 52} (1950), 650--665.

\bibitem{mks} W.\,Magnus, A.\,Karrass and D.\,Solitar,
{\it Combinatorial Group Theory},
Dover Publications, Inc., Mineola, New York 2004.

\bibitem{mios} A.\,Minaysan and D.\,Osin,
{Acylindrical hyperbolicity of groups acting on trees},
\texttt{http://arxiv.org/1310.6289}

\bibitem{pmn} P.\,M.\,Neumann,
{\it The SQ-universality of some finitely presented groups},
J. Austral. Math. Soc. {\bf 16} (1973) 1--6.

\bibitem{nibloreeves} G.\,A.\,Niblo, L.\,D.\,Reeves,
{\it The geometry of cube complexes and the complexity of their fundamental groups},
Topology {\bf 37} (1998), no. 3, 621--633. 

\bibitem{ol} A.\,Yu.\,Ol'shanski\u\i,
{\it SQ-universality of hyperbolic groups},
Sb. Math. {\bf 186} (1995) 1199--1211.

\bibitem{os04} D.\,Osin, 
Relatively hyperbolic groups: Intrinsic
geometry, algebraic properties, and algorithmic problems. {\it
Memoirs AMS} {\bf 179} (2006), no. 843.

\bibitem{os}
D.\,Osin, Acylindrically hyperbolic groups. Preprint (2013).\\  
\texttt{http://arxiv.org/1304.1246}

\bibitem{sacsch}
G.\,S.\,Sacerdote, P.\,E.\,Schupp, {\it SQ-universality in 
HNN groups and one relator groups}, J. London Math. Soc. \textbf{7} 
(1974), 733-740.


\bibitem{weh} B.\,A.\,F.\,Wehrfritz, {\it Generalized free products of
linear groups}, Proc. London Math. Soc. {\bf 27} (1973), 402--424. 

\bibitem{wslng} D.\,T.\,Wise,
{\it From Riches to Raags: 3-Manifolds, Right-Angled Artin
Groups, and Cubical Geometry},
CBMS Regional
Conference Series in Mathematics No. 117,
American Mathematical Society, Providence, RI, 2012.


\bibitem{wst} D.\,T.\,Wise,
{\it Cubular tubular groups},
Trans. Amer. Math. Soc. {\bf 366} (2014) 5503--5521.

\end{thebibliography}
\end{document}